\documentclass[11pt]{amsart}
\usepackage{latexsym,color}
\usepackage{amsthm}
\usepackage{amsmath}
\usepackage{amssymb}
\usepackage{mathtools}
\usepackage[T1]{fontenc}
\usepackage[utf8]{inputenc}
\newtheorem{theorem}{Theorem}

\newtheorem{cor}[theorem]{Corollary}
\newtheorem{rmk}[theorem]{Remark}
\newtheorem{expl}[theorem]{Example}

\newtheorem{conj}[theorem]{Conjecture}
\def \n{\noindent }
\def \bs{\bigskip}
\def \R{\mathbb R}
\def \Z{\mathbb Z}

\def \H{\mathcal H}

\def \S{\mathcal S}
\def \W{\mathcal W}

\def \O{\mathcal O}
\def \gc{\mathcal G}

\def \F{\mathcal F}
\def \ld{\lambda}
\def \sg{\sigma}
\def \al{\alpha}

\def \de{\delta}
\def \ta{\tau}
\def \bal{\boldsymbol\al}
\def \ssb{\sqsubseteq}

\def \w{{\bf w}}
\def \e{{\bf e}}

\def \a{{\bf a}}
\def \b{{\bf b}}

\def \x{{\bf x}}
\def \y{{\bf y}}
\def \z{{\bf z}}

\def \w{{\bf w}}
\def \e{{\bf e}}
\def \g{{\bf g}}
\def \o{{\bf 0}}

\def \conv{{\rm conv}}

\begin{document}
\title
[Geometry of the minimal solutions  of a linear Diophantine Equation]
{Geometry of the minimal solutions of a linear Diophantine Equation}
\author{Papa A. Sissokho}
\address{4520 Mathematics Department\\Illinois State University
\\ Normal, Illinois 61790--4520, U.S.A.}
\date{\today}
%
\begin{abstract}
Let $a_1,\ldots,a_n$ and $b_1,\ldots,b_m$ be fixed positive integers, and let $\S$ denote the set of all nonnegative integer solutions of the equation $x_1a_1+\ldots +x_na_n=y_1b_1+\ldots +y_mb_m$. A solution $(x_1,\ldots,x_n,y_1,\ldots,y_m)$ 
in $\S$ is called {\em minimal} if it cannot be expressed as the sum of two nonzero solutions in $\S$.  
For each pair $(i,j)$ with $1\leq i\leq n$ and $1\leq j\leq m$, the  solution 
whose only nonzero coordinates are $x_i=b_j$ and $y_j=a_i$ is called a {\em generator}. 
Our main result shows that every minimal solution is a convex combination of the generators and the zero-solution. 
This proves a conjecture of Henk--Weismantel and, independently,  Hosten--Sturmfels. 
\end{abstract}
\maketitle
\section{Introduction and main results}
\subsection{Introduction}
For any integer $t$, let $\Z_{\geq t}\coloneqq\{x\in\Z:\; x\geq t\}$ and define $\R_{\geq t}$ in a similar manner. Given two vectors $\a=(a_1,\ldots,a_n)\in\Z_{\geq1}^n$ and $\b=(b_1,\ldots,b_m)\in\Z_{\geq1}^m$, let $\S(\a,\b)$ denote the set of all solutions $(\x,\y)\in \Z_{\geq0}^n\times \Z_{\geq0}^m$, with $\x=(x_1,\ldots,x_n)$ and $\y=(y_1,\ldots,y_m)$, of the linear Diophantine equation 
\begin{equation}\label{eq:lde}
x_1a_1+\ldots +x_na_n=y_1b_1+\ldots+y_mb_m.
\end{equation}
For any nonzero solution $(\x,\y)\in\S(\a,\b)$, define
\begin{equation}\label{eq:supp}
{\rm supp}(\x):=\{i:\, x_i>0,\; 1\leq i\leq n\}; \; 
{\rm supp}(\y):=\{j:\, y_j>0,\; 1\leq j \leq m\}.
\end{equation}

A solution is called {\em minimal} if it cannot be 
written as the sum of two nonzero solutions in $\S(\a,\b)$. The set of all minimal solution of~\eqref{eq:lde}, denoted by $\H(\a,\b)$, is the 
{\em Hilbert basis} of the pointed rational cone 
\[\{(\x,\y)\in \R_{\geq0}^{n}\times\R_{\geq0}^m:\; \a\cdot \x=\b\cdot \y\},\]
where $\a\cdot\x$ denotes the dot product of $\a$ and $\x$ (similarly for $\b\cdot\y$).

For any integer $k$ with $1\leq k\leq n+m$, 
let $\e_k$ denote the $k$th standard unit vector of $\R^{n+m}$.   If $1\leq i\leq n$ and $1\leq j\leq m$, then $\g_{i,j}=(b_j\e_i,a_i\e_{n+j})$ is a solution of~\eqref{eq:lde} called {\em generator}. A generator $\g_{i,j}$ is minimal if and only if $\gcd(a_i,b_j)=1$. In particular, if $d_{i,j}=\gcd(a_i,b_j)$, then $(1/d_{i,j})\cdot \g_{i,j}$ is minimal.
Let $\W(\a,\b)$ denote the convex hull of the 
zero-solution and the elementary solutions $\g_{i,j}$, 
i.e., 
\begin{equation}\label{eq:cvx}
\W(\a,\b)={\rm conv}\big( \{\o\}\cup\left\{\g_{i,j}:\,\mbox{$1\leq i\leq n$, $1\leq j\leq m$}\right\}\big).
\end{equation}
The following conjecture was made by Henk and Weismantel~\cite[Conjecture~$1$]{HW}, and, independently\footnote{This information was given by Henk and Weismantel~\cite[Page 54]{HW}}, by Hosten and Sturmfels.
\begin{conj}\label{conj:1}
If $\a\in\Z_{\geq1}^n$ and $\b\in\Z_{\geq1}^m$, then $\H(\a,\b)\subseteq \W(\a,\b)$. 
\end{conj}
For $n=1$ or $m=1$, it was noted in~\cite{HW} that the Conjecture~\ref{conj:1} follows from a theorem of  Lambert~\cite{La} (and, independently, by Diaconis--Graham--Sturmfels~\cite{DGS}), which states that if $(\x,\y)$ is a minimal solution, then 
\begin{equation}\label{eq:Lam}
||\x||_1=\sum_{i=1}^n x_i \leq \max_{1\leq j\leq m} b_j \mbox{ and }
||\y||_1=\sum_{j=1}^my_j\leq \max_{1\leq i\leq n} a_i.
\end{equation}
The above upper bounds have been subsequently  
improved by Henk and Weismantel~\cite{HW}. Currently, 
the best known upper bounds (given in~\cite{S}) are
\begin{equation}\label{eq:Sis}
||\x||_1 \leq \frac{\y\cdot \b}{||\y||_1} \mbox{ and }
||\y||_1 \leq \frac{\x\cdot \a}{||\x||_1}.
\end{equation}

\bs
The rest of the paper is organized as follows. In Section~\ref{sec:main}, we present our main theorem (Theorem~\ref{thm:gen}) whose immediate corollary is the proof of Conjecture~\ref{conj:1}. In Section~\ref{sec:Grv}, we 
use Theorem~\ref{thm:gen} to characterize the {\em Graver Basis} of matrices with a single row. Then in Section~\ref{sec:com}, we apply Theorem~\ref{thm:gen} to 
{\em completely fundamental solutions}. Finally, we discuss the algorithmic nature of the proof of Theorem~\ref{thm:gen} 
and illustrate it with an example.  
\subsection{Main theorem}\label{sec:main}\
 In this section, we prove the following theorem  
and use it to verify Conjecture~\ref{conj:1}.
\begin{theorem}\label{thm:gen}
Let $\a=(a_1,\ldots,a_n)\in\Z_{\geq1}^n$ and $\b=(b_1,\ldots,b_m)\in\Z_{\geq1}^m$. If $(\x,\y)\in\H(\a,\b)$, with $\x=(x_1,\ldots,x_n)$ and $\y=(y_1,\ldots,y_m)$, then   
there exist rational numbers $\ld_{i,j}$ such that
\begin{equation}\label{eq:alij}
\begin{cases}
\displaystyle{\sum\limits_{j=1}^m \ld_{i,j}\cdot b_j=x_i},\;&\mbox{ for $1\leq i\leq n$}\\
\displaystyle{\sum_{i=1}^n\ld_{i,j}\cdot a_i=y_j,\;
}&\mbox{ for $1\leq j\leq m$}\\
\displaystyle{\sum_{i=1}^n\sum_{j=1}^m \ld_{i,j}\leq 1}\\
\ld_{i,j}\geq 0,\;&\mbox{ for $1\leq i\leq n$ and $1\leq j\leq m$}.
\end{cases}
\end{equation}
\end{theorem}

Conjecture~\ref{conj:1} immediately follows from Theorem~\ref{thm:gen} by setting the coefficient of
$\o$ to $1-\sum\limits_{i=1}^n\sum\limits_{j=1}^m\ld_{i,j}$. More precisely, we have the following corollary 
whose (short) proof is in Section~\ref{sec:proof}.
\begin{cor}\label{cor:2}  Every minimal solution $(\x,\y)$ is a convex combination of $\o$ and the generators.
Moreover, one can use up to $m+n-1$  nonzero generators for any such combination.
\end{cor}
\begin{rmk}
Carath\'eodory's theorem already predicts a convex combination of $(\x,\y)$ 
with at most $d+1$ vertices from the convex hull $\W(\a,\b)$ given in~\eqref{eq:cvx}, where $d=\dim\,\W(\a,\b)=n+m-1$.
\end{rmk}

\bs It was mentioned in~\cite[Page 54]{HW} that Hosten and Sturmfels found an example for which the convex hull of $\o$ and the minimal generators $\left((b_j/d_{i,j})\e_i,(a_i/d_{i,j})\e_{n+j}\right)$, with $d_{i,j}=\gcd(a_i,b_j)$, does not contain the Hilbert basis $\H(\a,\b)$. Since that ``example'' was not included in their article, we give another example of that fact below.
\begin{expl}\label{exp1}
Let $\a=(6)$ and  $\b=(2,3,5)$. Then 
$d_{1,1}=\gcd(6,2)=2$, $d_{1,2}=\gcd(6,3)=3$, and $d_{1,3}=\gcd(6,5)=1$.  
Consider the minimal solution $(\x,\y)=\left((2),(2,1,1)\right)$ of the equation $6x_1=2y_1+3y_2+5y_3$.
To obtain a convex hull that involves the minimal generators, we need to find a solution $\left(\ld_{1,1},\ld_{1,2},\ld_{1,3}\right)\in\R^{3}$ to the following linear system:
\[
\begin{cases}
\frac{a_1}{d_{1,1}}\ld_{1,1}=y_1 \cr
\frac{a_1}{d_{1,2}}\ld_{1,2}=y_2\cr
\frac{a_1}{d_{1,3}}\ld_{1,3}=y_3\cr
\frac{b_1}{d_{1,1}}\ld_{1,1}+\frac{b_2}{d_{1,2}}\ld_{1,2}+\frac{b_3}{d_{1,3}}\ld_{1,3}=x_1 \cr
\ld_{1,1}+\ld_{1,2}+\ld_{1,3}\leq 1\cr
\ld_{1,1},\, \ld_{1,2},\, \ld_{1,3}\geq 0.
\end{cases}
\Longleftrightarrow
\quad 
\begin{cases}
3\ld_{1,1}=2 \cr
2\ld_{1,2}=1\cr
6\ld_{1,3}=1\cr
\ld_{1,1}+\ld_{1,2}+5\ld_{1,3}=2 \cr
\ld_{1,1}+\ld_{1,2}+\ld_{1,3}\leq 1\cr
\ld_{1,1},\, \ld_{1,2},\, \ld_{1,3}\geq 0.
\end{cases}
\]
The only solution to the subsystem composed by the first four equations is $(\ld_{1,1},\ld_{1,2},\ld_{1,3})=(\frac{2}{3},\frac{1}{2},\frac{1}{6})$. Thus, $\ld_{1,1}+\ld_{1,2}+\ld_{1,3}=\frac{4}{3}$, which violates the fifth constraint.
\end{expl}

\bs However, in contrast to the situation illustrated in Example~\ref{exp1}, we have the following interesting corollary. 
\begin{cor}\label{cor:added1}  If $\gcd(a_i,b_j)=1$ for all $1\leq i\leq n$ and $1\leq j\leq m$, then the 
generators $\g_{i,j}=(b_j\e_i,a_i\e_{n+j})$ are all minimal and the set of extreme points of  $\conv\left(\H(\a,\b)\right)$ 
is 
\[ \{\o\}\cup\left\{\g_{i,j}:\,\mbox{$1\leq i\leq n$, $1\leq j\leq m$}\right\}.\]
\end{cor}
\begin{proof}
Since $\H(\a,\b)$ is the set of minimal solutions, it follows from the first part of Corollary~\ref{cor:2} that
\[ \H(\a,\b)\subseteq {\rm conv}\big( \{\o\}\cup\left\{\g_{i,j}:\,\mbox{$1\leq i\leq n$, $1\leq j\leq m$}\right\}\big),
\]
which implies that 
\begin{equation}\label{eq:hb1}
\conv\left(\H(\a,\b)\right)\subseteq {\rm conv}\big( \{\o\}\cup\left\{\g_{i,j}:\,\mbox{$1\leq i\leq n$, $1\leq j\leq m$}\right\}\big).
\end{equation}
Moreover, since  $\gcd(a_i,b_j)=1$ for all $1\leq i\leq n$ and $1\leq j\leq m$, the elements of $\{\o\}\cup\left\{\g_{i,j}:\,\mbox{$1\leq i\leq n$, $1\leq j\leq m$}\right\}$ are  minimal and belong to the Hilbert basis $\H(\a,\b)$. Thus,  
\begin{equation}\label{eq:hb2}
{\rm conv}\big( \{\o\}\cup\left\{\g_{i,j}:\,\mbox{$1\leq i\leq n$, $1\leq j\leq m$}\right\}\big) \subseteq \conv\left(\H(\a,\b)\right).
\end{equation}
By combining~\eqref{eq:hb1} and~\eqref{eq:hb2}, we obtain
\[
\conv\left(\H(\a,\b)\right)= {\rm conv}\big( \{\o\}\cup\left\{\g_{i,j}:\,\mbox{$1\leq i\leq n$, $1\leq j\leq m$}\right\}\big).\]
Finally, it follows from the above equality and the minimality of the generators $\g_{i,j}$ and $\o$ that they are the extreme points of $\conv\left(\H(\a,\b)\right)$.
\end{proof}

\bs
Before ending this section, we present a useful way of viewing a solution $(\x,\y)\in \S(\a,\b)$ as a
{\em partition identity}, which is an equality of the form 
\[\underbrace{a_1+\ldots+a_1}_{\mbox{$x_1$ copies}}+\ldots+ \underbrace{a_n+\ldots+a_n}_{\mbox{$x_n$ copies}}= \underbrace{b_1+\ldots+b_1}_{\mbox{$y_1$ copies}}+\ldots+\underbrace{b_m+\ldots+b_m}_{\mbox{$y_m$ copies}}\]
where we skip a term $a_i$ (resp. $b_j$) if $x_i=0$ (resp. $y_j=0$). A partition identity is called {\em primitive} if it does not contain 
a proper nonempty subpartition identity. For instance, $2+2+3=2+5$ is not primitive since it contains the subpartition identity $2+3=5$.  Primitive Partition identities were introduced by Diaconis--Graham--Sturmfels~\cite{DGS}, where their relevance and applications to several areas were demonstrated. In the proof of our main theorem (see Section~\ref{sec:proof}), we sometimes view minimal solutions  in $\H(\a,\b)$ as primitive partition identities. 

\subsection{Graver Basis}\label{sec:Grv}

Let $r$ and $k$ be positive integers. For each $\ta=(\ta_1,\ldots,\ta_k)
\in\{-1,1\}^k$, we can associate the following orthant of $\R^k$ 
\[\O^{(\ta)}=\{\x\in \R^k:\; \ta_ix_i\geq0\mbox{ for } 1\leq i\leq k\}.\]
Let $A$ be an $r\times k$ matrix with entries in 
$\Z$ and define the relation $\ssb$ on $D=\{\x\in\Z^k:\; A\x=\o \mbox{ and } \x\not=\o\}$ as follows. For any $\x=(x_1,\ldots,x_k)$ and $\x'=(x'_1,\ldots,x'_k)$ in $D$, we write $\x'\ssb \x$ if these two vectors are in the same orthant of $\R^k$ and $|x'_i|\leq |x_i|$ for $1\leq i\leq k$. For instance $(1,-2,3)\ssb (1,-3,3)$, but $(1,-2,-3)$
and $(2,1,2)$ are not comparable since they live in different orthants of $\R^3$.
We say that $\x\in D$ is $\ssb$-minimal if there is no $\x'\in D$ such that  $\x'\ssb\x$.  
The {\em Graver basis} of $A$, denoted by $\gc(A)$, 
is the set of all $\ssb$-minimal vectors in $D$.

The concept of a Graver basis was introduced by Graver~\cite{Gr} as a method for solving certain classes of linear and integer optimization problems. This has since been extended to a wider class of problems along with polynomial-time (in the size of the inputs) algorithms (e.g., see~\cite[Chapter 3]{DHK}).
If $\H^{(\ta)}$ denote the Hilbert basis of the 
pointed cone $\O^{(\ta)}\cap \{\x\in\R^k:\; A\x=\o\}$, 
then it is well-known that 
\begin{equation}\label{eq:gr}
\gc(A)=\bigcup_{\ta\in\{-1,1\}^n}\H^{(\ta)}\setminus\{\o\}.
\end{equation}
Because of this direct relationship between the Hilbert 
basis and the Graver basis, our main theorem yields  
the following corollary when $A$ has a single row. 
\begin{cor}\label{cor:5} 
Suppose $A$ has a single row, i.e., $A=\bal=(\al_1,\ldots,\al_k)\in\Z^k$.

\n $(i)$  For each orientation $\ta=(\ta_1,\ldots,\ta_k)
\in\{-1,1\}^k$, we have $\H^{(\ta)}\subseteq\conv\left(\F^{(\ta)}\right)$, 
where
\[\F^{(\ta)}=\{\o\}\cup \{\ta_i|\al_j|\e_i+\ta_j|\al_i|\e_{j}:\;1\leq i\not=j\leq k,\;(\ta_i\al_i)(\ta_j\al_j)<0\}.\]

\n $(ii)$  $\gc(\bal)\subseteq\conv\left(\F\right)$, 
where
\[\F=\{\o\}\cup \{\sg|\al_j|\e_i+\de|\al_i|\e_{j}:\;1\leq i\not=j\leq k,\;\sg, \de\in\{-1,1\},\; \sg\al_i\de\al_j<0\}.\]
\end{cor}
\begin{rmk}
Corollary~\ref{cor:5} has applications in Integer Programming for certain families of knapsack problems. These applications will be discussed elsewhere.
\end{rmk}
We delay the proof until the end of Section~\ref{sec:proof} and consider the following example.

\begin{expl}\label{exp_gr}
Suppose $\al=(1,2,-3)$, then the main equation is 
\begin{equation}\label{eq:exp}
x_1+2x_2-3x_3=0.
\end{equation}
For each orientation $\ta=(\ta_1,\ta_2,\ta_3)\in\{-1,1\}^3$ (which correspond to some orthant or $\R^3$), the equation in~\eqref{eq:exp} over the domain  
\[D^{(\ta)}=\{(x_1,x_2,x_3)\in\Z^3:\; \ta_ix_i\geq 0\mbox{ for $1\leq i\leq 3$}\}\] 
can be solved by first considering the equation 
\[\ta_1z_1+2\ta_2z_2-3\ta_3z_3=0\]
over $\Z_{\geq0}^3$. Then, each solution $(z_1,z_2,z_3)\in \Z_{\geq0}^3$
to this latter equation gives rise to a solution 
$\x=(\ta_1z_1,\ta_2z_2,\ta_3z_3)\in D^{(\ta)}$. We summarize this in Table~\ref{tab_gr} below where we handle $\ta$ and $-\ta$ together as $\pm\ta$ since $\x\in \gc(A)$
implies that $-\x\in\gc(A)$.

\begin{table}[ht]
\caption{Finding the sets $\F^{(\ta)}$ with $\ta=(\ta_1,\ta_2,\ta_3)\in\{-1,1\}^3$.}
\begin{center}
\begin{tabular}{|p{3.6cm}| p{3.1cm} | p{4.9cm} | p{4.5cm} |}
\hline
\mbox{\bf Equation over $\Z_{\geq0}^3$}& \mbox{\bf $\pm\ta=\pm (\ta_1,\ta_2,\ta_3)$} & \mbox{\bf $\z\in \Z_{\geq0}^3$
$\longrightarrow\pm\x\in D^{(\pm\ta)}$}   & \centerline{\mbox{\bf $\F^{(\pm\ta)}$}} \\ \hline 
\hline 	
$z_1+2z_2=3z_3$ &$\pm(1,1,1)$ & $(1,1,1) \to \pm(1,1,1) $ & $\pm \{\o,3\e_1+\e_3,3\e_2+2\e_3\}$ \\ \hline
$z_1+2z_2+3z_3=0$ & $\pm(1,1,-1)$ & \mbox{no nonzero solution in $\Z_{\geq0}^3$} & $\;\;\;\{\o\}$
 \\ \hline
$z_1=2z_2+3z_3$ & $\pm (1,-1,1)$ &  $(5,1,1)\to \pm(5,-1,1)$ & $\pm \{\o,2\e_1-\e_2,3\e_1+\e_3\}$  \\ \hline
$z_1+3z_3=2z_2$ &  $\pm (1,-1,-1)$ &  $\pm (1,2,1)\to \pm (1,-2,-1)$ & $\pm\{\o,2\e_1-\e_2,-3\e_2-2\e_3\}$ \\ \hline
\end{tabular}
\end{center}
\label{tab_gr}
\end{table}
Note that the solutions listed in the third column of Table~\ref{tab_gr} are just examples of solutions 
and they are not necessarily $\ssb$-minimal. Moreover, it follows from the last column of Table~\ref{tab_gr} that
\[\F=\bigcup_{\ta\in\{-1,1\}^n}\F^{(\ta)}=\{\o,\pm(3\e_1+\e_3),\pm(3\e_2+2\e_3),\pm(2\e_1-\e_2)\}.\]
\end{expl}

\subsection{Completely Fundamental Solutions}\label{sec:com}
A solution $(\x,\y)\in\S(\a,\b)$ is called {\em completely fundamental} if for every decomposition $k(\x,\y)=(\x',\y')+(\x'',\y'')$ with $k\in \Z_{\geq1}$, $(\x',\y')\in\S(\a,\b)$, and $(\x'',\y'')\in\S(\a,\b)$, there exist nonnegative integers $r$ and $s$ such that
\[(\x',\y')=r(\x,\y),\;(\x'',\y'')=s(\x,\y), 
\;\mbox{ and } r+s=k.\]
Let $\S_c(\a,\b)$ denote the set of all completely fundamental solutions in $\S(\a,\b)$. In particular  
$\S_c(\a,\b)\subseteq \H(\a,\b)$.  
Completely fundamental solutions were introduced by Stanley~\cite{St} who used them to characterize  
 a certain generating function associated with 
the Hilbert basis of a system of linear homogeneous Diophantine equations. In the case of a single equation, 
that this paper is concerned with, the generating function in question is as follows:
\begin{equation}\label{genf}
F_{\a,\b}(\z,\w)= F_{\a,\b}(z_1,\ldots,z_n,w_1,\ldots,w_m)=\sum_{(\x,\y)\in\S}z_1^{x_1}\ldots z_n^{x_n}\cdot w_1^{y_1}\ldots w_m^{y_m}=\sum_{(\x,\y)\in\S}\z^{\x}\,\w^{\y},
\end{equation}
where $\S=\S(\a,\b)$, $\x=(x_1,\ldots,x_n)$ and $\y=(y_1,\ldots,y_m)$.
In the next corollary (of Theorem~\ref{thm:gen}), we 
determine the completely minimal solutions in $\S(\a,\b)$. Note that we only use the fact that 
the $\ld_{i,j}$ coefficients in Theorem~\ref{thm:gen} are rational numbers and not that their sum is at most $1$.
\begin{cor}\label{cor:3}
A solution $(\x,\y)\in\S(\a,\b)$ is completely fundamental  if and only if $(\x,\y)$ is a minimal generator; i.e., 
\[
\S_c(\a,\b)=\{\left((b_j/d_{i,j})\e_i,(a_i/d_{i,j})\e_{n+j}\right):\, 1\leq i\leq n,\;1\leq j\leq m\},\]
where $d_{i,j}=\gcd(a_i,b_j)$.
\end{cor}

We can now deduce the following corollary from Corollary~\ref{cor:3} and a result of Stanley~\cite[Theorem~$2.5$]{St}.
\begin{cor}\label{cor:4}
If $\S_c=\S_c(\a,\b)$ be the set of all completely fundamental solutions in $\S(\a,\b)$, then 
$F_{\a,\b}(\z,\w)$ is a rational function with denominator 
\[D_{\a,\b}(\z,\w)=\prod_{(\x,\y)\in\S_c}(1-\z^{\x}\,\w^{\y})=\prod_{\left((b_j/d_{i,j})\e_i,(a_i/d_{i,j})\e_{n+j}\right)\in\S_c}\left(1-z_i^{b_j/d_{i,j}}\,w_j^{a_i/d_{i,j}}\right).\]
\end{cor}
%
\section{Proof of the main results}\label{sec:proof}
In this section, we prove Theorem~\ref{thm:gen}, Corollary~\ref{cor:2}, and Corollary~\ref{cor:5}.
\begin{proof}[Proof of Theorem~\ref{thm:gen}]
Fix the coefficients $\a$ and $\b$ and consider a (nonzero) minimal solution $(\x,\y)\in \H(\a,\b)$. The proof is by induction on $ ||\x||_1+ ||\y||_1$, which is at least $2$.
We first argue that one may assume (without loss of generality) that $\x$ and $\y$ are binary vectors. To see this, first 
note that from~\eqref{eq:Lam} implies  $x_i\leq ||\x||_1\leq \max_{1\leq j\leq m} b_j$ for $1\leq i\leq n$, and $y_j\leq ||\y||_1\leq \max_{1\leq i\leq n} a_i$ for $1\leq j\leq m$. For each $i$ such that $x_i\geq1$, let   
$a'_{i,k}=a_i$ for $1\leq k\leq x_i$ (i.e., we create $x_i$ copies of $a_i$), and for each $j$ such that $y_j\geq 1$,  let $b'_{j,s}=b_j$ for $j\leq s\leq y_j$. Now define $p:=\sum_{i=1}^nx_i= ||\x||_1$, $q:=\sum_{j=1}^my_j= ||\y||_1$, 
\[\a'=(a'_{1,1},\ldots,a'_{1,x_1},\ldots,a'_{n,1},\ldots,a'_{n,x_n})\in\Z_{\geq0}^{p},\] 
and 
\[\b'=(b'_{1,1},\ldots,b'_{1,y_1},\ldots,b'_{m,1},\ldots,b'_{m,y_m})\in\Z_{\geq0}^{q}.\]
Then the pair $\left((1,\ldots,1),(1,\ldots,1)\right)\in\Z_{\geq0}^{p+q}$ is a minimal nonnegative solution of the Diophantine equation 
\begin{equation}\label{eq:lde2}
\begin{cases}
\displaystyle{\sum_{i=1}^n\sum_{k=1}^{x_i}x'_{i,k}a'_{i,k}=
\sum_{j=1}^m\sum_{\ell=1}^{y_j}y'_{j,\ell}b'_{i,k}}\\
x'_{i,k},y'_{j,\ell}\in\{0,1\},
\end{cases}
\end{equation}
which is equivalent to the solution $(\x,\y)$ that we started with. 
Since $||\x||_1+||\y||_1$ will decrease at each step of the inductive step, the transformed dimension $p+q=||\x||_1+||\y||_1$ 
is also decreasing. Moreover, as we shall see, the maximum entry in any generator $\g_{i,j}$, used in the convex combination of $(\x,\y)$, will not increase at any stage of the inductive step because the coefficients (originally $\a$ and $\b$) will not increase.
Thus, it suffices to prove the theorem for linear Diophantine equations with binary variables. In particular, we may assume, without loss of generality, that for any minimal solution $(\x,\y)\in\H(\a,\b)$, we have 
$x_i=1$ if $i\in{\rm supp}(\x)$ and $x_i=0$ if $i\not\in{\rm supp}(\x)$ (and similarly for $y_j$). 
Let $s$ and $t$ be such that 
\[a_s=\max_{i\in {\rm supp}(\x)} a_i \mbox { and } b_t=\max_{j\in {\rm supp}(\y)}b_j.\]

For the base case, let $||\x||_1+||\y||_1=2$. Then  $||\x||_1=||\y||_1=1$,  ${\rm supp}(\x)=\{s\}$, ${\rm supp}(\y)=\{t\}$, $a_s=b_t$, and $(\x,\y)=(\e_s,\e_t)$. If we set $\ld_{s,t}=\frac{1}{a_s}=\frac{1}{b_t}$ and $\ld_{i,j}=0$ for $(i,j)\not=(s,t)$,
then $\ld_{s,t}\leq 1$ and the required constraints
in~\eqref{eq:alij} are clearly satisfied.   
 
Next, assume that $||\x||_1+||\y||_1>2$. Then, it follows from the minimality of $(\x,\y)$ that $a_s\not=b_t$. Without loss of generality, assume that $a_s>b_t$. Let $a_{n+1}=a_s-b_t$ (thus, the new coefficient $a_{n+1}$ is decreasing) and consider the vectors $\x'=(x'_1,\ldots,x'_{n+1})$ and $\y'=(y'_1,\ldots,y'_m)$ given by 
\begin{equation}\label{eq:set}
\begin{cases}
\mbox{$x'_{n+1}=1$, $x'_s=0$,  $x_i'=x_i$ for $1\leq i\leq n$ and $i\not=s$},\\
\mbox{$y'_t=0,\;y_j'=y_j$ for $1\leq j\leq m$ and $j\not=t$}.
\end{cases}
\end{equation}
Then it can easily be seen that $(\x',\y')$ is also a minimal solution of the Diophantine equation 
\[
\begin{cases}
z_1a_1+\ldots+ z_{n+1}a_{n+1}=w_1b_1+\ldots +\ldots+w_{m}b_{m}\\
z_i,w_j\in\{0,1\}
\end{cases}
\]
Moreover, we have
\[{\rm supp}(\x')=\left({\rm supp}(\x)\setminus\{s\}\right)\cup\{n+1\},\; {\rm supp}(\y')={\rm supp}(\y)\setminus\{t\},\;  ||\x'||_1=||\x||_1,\mbox{ and }  ||\y'||_1=||\y||_1-1.\]
Hence, $2\leq ||\x'||_1+||\y'||_1<||\x||_1+||\y||_1$, and it follows from the induction hypothesis that there exist nonnegative rational numbers $\ld'_{i,j}$ such that, 
\begin{equation}\label{eq:IH1}
\ld'_{i,j}=0,\,\mbox{ if }(i,j)\not\in {\rm supp}(\x')\times{\rm supp}(\y'),
\end{equation}
\begin{equation}\label{eq:IH2}
\begin{cases}
\displaystyle{\sum_{j=1}^m\ld'_{i,j}b_j=x'_i},\, &\mbox{ if $1\leq i\leq n+1$}\\
\displaystyle{\sum_{i=1}^{n+1}\ld'_{i,j}a_i=y'_j},\, &\mbox{ if $1\leq j\leq m$}.
\end{cases}
\end{equation}
\begin{equation}\label{eq:IH3}
\displaystyle{\sum_{i=1}^{n+1}\sum_{j=1}^m \ld'_{i,j}\leq 1}.
\end{equation}
Since $t\not\in{\rm supp}(\y')$, it follows from~\eqref{eq:IH1} that $\ld'_{n+1,t}=0$. This fact, together with~\eqref{eq:set} and~\eqref{eq:IH2} imply that
\begin{align}\label{eq:lb1}
\sum_{j=1,\,j\not=t}^m\ld'_{n+1,j}b_j=\sum_{j=1}^m\ld'_{n+1,j}b_j=x'_{n+1}=1.
\end{align}
Since $b_t=\max\limits_{j\in{\rm supp}(\y)}b_j\geq1$ and ${\rm supp}(\y')\subseteq {\rm supp}(\y)$, it follows from~\eqref{eq:lb1} that 
\begin{align}\label{eq:lb2}
&b_t\sum_{j=1,\,j\not=t}^m\ld'_{n+1,j}\geq \sum_{j=1,\,j\not=t}^m\ld'_{n+1,j}b_j =1\cr
&\Longrightarrow \sum_{j=1,\,j\not=t}^m\ld'_{n+1,j}\geq \frac{1}{b_t}.
\end{align}
\bs
We now define $\{\ld_{i,j}\}_{i=1,j=1}^{n,m}$ as follows:
\begin{equation}\label{eq:al}
\begin{cases}
\ld_{s,t}\coloneqq\frac{1}{a_s}\\
\ld_{s,j}\coloneqq\frac{a_{n+1}}{a_{n+1}+b_t}\cdot \ld'_{n+1,j} =\frac{a_s-b_t}{a_s}\cdot \ld'_{n+1,j},
&\mbox{ if $1\leq j\leq m$ and $j\not=t$}\\
\ld_{i,t}\coloneqq0=\ld'_{i,t}\mbox{ (since $t\not\in{\rm supp(\y')}$)}, &\mbox{ if $1\leq i\leq n$ and $i\not=s$}\\
\ld_{i,j}\coloneqq\ld'_{i,j}, &\mbox{ if $1\leq i\leq n$, $1\leq j \leq m$, $i\not=s$, and $j\not=t$}.
\end{cases} 
\end{equation}
Thus, the $\ld_{i,j}$'s are nonnegative rational numbers.  It also follows from~\eqref{eq:set},~\eqref{eq:IH1}, and~\eqref{eq:al} that 
\begin{equation}\label{eq:non_supp}
\ld_{i,j}=0,\,\mbox{ if }(i,j)\not\in {\rm supp}(\x)\times{\rm supp}(\y).
\end{equation} 
By induction hypotheses in~\eqref{eq:IH1}--\eqref{eq:IH2} and the definition of $\ld_{i,j}$ 
in~\eqref{eq:al}, it follows that if $1\leq i\leq n$ and $i\not=s$, then 
\[
\sum_{j=1}^m\ld_{i,j}b_j=\sum_{j=1}^m\ld'_{i,j}b_j=x'_i=x_i.
\]
If $1\leq j\leq m$ and $j\not=t$, we have 
\begin{align*}
\sum_{i=1}^{n}\ld_{i,j}a_i
&=\sum_{i=1,\,i\not=s}^{n}\ld_{i,j}a_i+\ld_{s,j}a_s\\
&=\sum_{i=1,i\not=s}^{n}\ld'_{i,j}a_i+\left(\frac{a_s-b_t}{a_s}\right)\ld'_{n+1,j}a_s\quad\mbox{(by~\eqref{eq:al})} \\
&=\sum_{i=1,i\not=s}^{n}\ld'_{i,j}a_i+\ld'_{n+1,j}(a_s-b_t)\\
&=\sum_{i=1,i\not=s}^{n}\ld'_{i,j}a_i+\ld'_{s,j}a_s+\ld'_{n+1,j}a_{n+1}\quad\mbox{(since $\ld'_{s,j}=0$ and $a_{n+1}=a_s-b_t$)} \\
&=\sum_{i=1}^{n+1}\ld'_{i,j}a_i\\
&=y'_j=y_j\quad\mbox{(by~\eqref{eq:IH2} and~\eqref{eq:set}).}
\end{align*}
Moreover,~\eqref{eq:IH1} and~\eqref{eq:al} yield
\[
\sum_{i=1}^{n}\ld_{i,t}a_i=\ld_{s,t}a_s=\frac{1}{a_s}\cdot a_s=1=y_t,
\]
and
\begin{align*}
\sum_{j=1}^m\ld_{s,j}b_j
&=\sum_{j=1,\,j\not=t}^m\ld_{s,j}b_j+\ld_{s,t}\cdot b_t\cr
&=\sum_{j=1,\,j\not=t}^m\left(\frac{a_s-b_t}{a_s}\cdot \ld'_{n+1,j}\right)b_j+\ld_{s,t}\cdot b_t\quad
\mbox{(by~\eqref{eq:al})} \cr
&=\frac{a_s-b_t}{a_s}\sum_{j=1}^m\ld'_{n+1,j}\cdot b_j+\frac{b_t}{a_s} \cr
&=\frac{a_s-b_t}{a_s}\,x'_{n+1}+\frac{b_t}{a_s}\quad\mbox{ (by~\eqref{eq:IH2})}\cr
&=1=x_s\quad\mbox{ (since $x'_{n+1}=1$ by~\eqref{eq:set})}.
\end{align*}
Still using~\eqref{eq:IH1}, \eqref{eq:al}, and~\eqref{eq:non_supp}, we infer that 
\begin{align}\label{eq:lb3}
&\sum_{i=1}^n\sum_{j=1}^m \ld_{i,j}-\sum_{i=1}^{n+1}\sum_{j=1}^m \ld'_{i,j}\cr
&=\left(\sum_{i=1,\,i\not=s}^n\ld_{i,t}-\sum_{i=1,\,i
\not=s}^n\ld'_{i,t}-\ld'_{n+1,t}\right)+\left(\sum_{j=1,
\,j\not=t}^m\ld_{s,j}-\sum_{j=1,\,j\not=t}^m\ld'_{n+1,j}
\right)+(\ld_{s,t}-\ld'_{s,t})\cr
&=\sum_{j=1,\,j\not=t}^m\ld_{s,j}+\ld_{s,t}-\sum_{j=1,\,j\not=t}^m\ld'_{n+1,j}\quad\mbox{(since $\ld'_{s,t}=0$, $\ld'_{n+1,t}=0$, and $\ld_{i,t}=\ld'_{i,t}$ if $i\not=s$)}\cr 
&= \sum_{j=1,\,j\not=t}^m\left(\frac{a_s-b_t}{a_s}\right)\ld'_{n+1,j}+\frac{1}{a_s}-\sum_{j=1,\,j\not=t}^m\ld'_{n+1,j}\cr
&=\frac{1}{a_s}\left(1-b_t\sum_{j=1,\,j\not=t}^m\ld'_{n+1,j}\right)\cr
&\leq 0,
\end{align}
where the last inequality follows from~\eqref{eq:lb2}.

Finally, it follows from~\eqref{eq:IH3} and~\eqref{eq:lb3} that 
\[\sum_{i=1}^n\sum_{j=1}^m \ld_{i,j}\leq \sum_{i=1}^{n+1}\sum_{j=1}^m \ld'_{i,j}\leq 1,\]
which completes the inductive step.
\end{proof}
\begin{proof}[Proof of Corollary~\ref{cor:2}]
This directly follows from the inductive proof of Theorem~\ref{thm:gen}. In particular, the inductive definition 
 in~\eqref{eq:IH1} and~\eqref{eq:non_supp} shows that at each step, $\ld_{i,j}\not=0$ for 
 exactly one new pair, namely $(i,j)=(s,t)$.  Then, as illustrated in Example~\ref{exp2}, we can switched back from the  transformed binary solution to the original (not necessarily binary) solution by combining all the generators in the binary solution that correspond to nonzero entry pairs from the original solution.
 
Alternatively\footnote{This alternative argument, which is more suited to the geometric theme of the paper,  was suggested by a referee of this paper.},  one can triangulate $\W(\a,\b)$ into $(n+m-1)$-dimensional simplices, all of which containing the vertex $\o$. Then it follows from Theorem~\ref{thm:gen} that each minimal solution belongs to one of those simplices. This yields the $m+n-1$ upper bound on the number of nonzero generators (vertices) used to represent a minimal solution as a convex combination of generators.
 
\end{proof}
\begin{proof}[Proof of Corollary~\ref{cor:5}]

To prove $(i)$, we proceed as in Example~\ref{exp2} in Section~\ref{sec:Grv}. Let $A=\bal=(\al_1,\ldots,\al_k)\in\Z^k$,  it suffices to prove 
that for each orientation $\ta=(\ta_1,\ldots,\ta_k)
\in\{-1,1\}^k$ and for any $\x\in \H^{(\ta)}$, we have 
$\x$ is a linear combination of vectors in the set 
\[\{\o\}\cup \{\ta_i|\al_j|\e_i+\ta_j|\al_i|\e_{j}:\;1\leq i\not=j\leq k,\;(\ta_i\al_i)(\ta_j\al_j)<0\}.\]

For any $\x=(x_1,\ldots,x_k)\in\H^{(\ta)}$, we have 
\begin{equation}\label{eq:prf_gr1}
\al_1x_1+\ldots+\al_k x_k=0.
\end{equation}
Since $\x$ is in the orthant $\O^{(\ta)}$, it follows that 
$\ta_ix_i\geq0$ for $1\leq i\leq k$. Since $\ta_i^2=1$ for $1\leq i\leq k$, we have
\[\al_1x_1+\ldots+\al_k x_k=0 \Longleftrightarrow \al_1\ta_1(\ta_1x_1)+\ldots+\al_k\ta_k (\ta_kx_k)=0.
\]
Thus, finding the solution $\x$ of ~\eqref{eq:prf_gr1} is equivalent to first find a nonnegative solution $\z=(z_1,\ldots,z_k)$ to the equation 
\begin{equation*}\label{eq:prf_gr2}
\ta_1\al_1z_1+\ldots+\ta_k\al_k z_k=0,
\end{equation*}
and then setting $x_i=\ta_iz_i$ for $1\leq i\leq k$. Moreover, the solution $\x\in\Z^k$ is $\ssb$-minimal if and only if 
the corresponding solution $\z\in \Z_{\geq0}^m$ is minimal. Thus, it follows from Theorem~\ref{thm:gen} that $\z$ is a convex
 combination of $\o$ and the generators of the equation in~\eqref{eq:prf_gr1}. By definition, these generators are the elements of the set
\[\{|\ta_j\al_j|\e_i+|\ta_i\al_i|\e_{j}:\;1\leq i\not=j\leq k,\;(\ta_i\al_i)(\ta_j\al_j)<0\}.\]
Since $\ta_i\in\{-1,1\}$ for $1\leq i\leq k$, we have $|\ta_i\al_i|=|\al_i|$, and $|\ta_j\al_j|=|\al_j|$. Thus,  since  $x_i=\ta_i z_i$ for $1\leq i\leq k$, it follows that $\x$ is a convex
combination of $\o$ and generators from the set 
\[\{\ta_i|\al_j|\e_i+\ta_j|\al_i|\e_{j}:\;1\leq i\not=j\leq k,\;(\ta_i\al_i)(\ta_j\al_j)<0\},\]
which proves $(i)$. 

To prove $(ii)$, we use part $(i)$ to obtain $\gc(\al)\subseteq \conv(\F)$, where 
\[\F=\bigcup_{\tau\in\{-1,1\}^n}\F^{(\tau)}=\{\o\}\cup \{\sg|\al_j|\e_i+\de|\al_i|\e_{j}:\;1\leq i\not=j\leq k,\;\ta, \de\in\{-1,1\},\; \sg\de\al_i\al_j<0\}\]
\end{proof}

\section{Algorithm for computing the $\ld_{i,j}$ in Theorem~\ref{thm:gen} via an Example}
The proof of Theorem~\ref{thm:gen} readily provides an algorithm for computing the coefficients $\ld_{i,j}$ in~\eqref{eq:alij} that  correspond to a given solution $(\x,\y)$. We illustrate the algorithm in Example~\ref{exp2} below.
\begin{expl}\label{exp2}
Consider the minimal solution  $(\x,\y)=\left((2),(2,1,1)\right)$ of the linear Diophantine equation $6x_1=2y_1+3y_2+5y_3$ from Example~\ref{exp1}. Then the corresponding (binary) linear Diophantine equation is
\[
\begin{cases}
6x_1+6x_2=2y_1+2y_2+3y_3+5y_4\\
x_i,\,y_j\in\{0,1\},
\end{cases}
\]
 with corresponding binary solution vectors $((1,1),(1,1,1,1))$. 

\begin{table}[ht]
\caption{Finding $a_s=\max_ia_i$ and $b_t=\max_jb_j$ at each level~$k$}
\begin{center}
\begin{tabular}{ |p{3.2cm}| p{5cm}| p{1.5 cm} | p{3.1cm} |}
\hline
\mbox{\bf Solution $(\x,\y)$} &\mbox{\bf Partition Identity} &\mbox{\bf Level $k$} & \mbox{\bf $a_s,\;b_t$ } \\  \hline 
\hline 	
 $\left((1,1),(1,1,1,1)\right)$ & $6+6=2+2+3+5$ & $1$ & $a_2=6,\;b_4=5$ \\ \hline
$\left((1,1),(1,1,1)\right)$ & $6+1=2+2+3$ & $2$ & $a_1=6,\;b_3=3$ \\ \hline
$\left((1,1),(1,1)\right)$ & $3+1=2+2$ & $3$ & $a_1=3,\;b_2=2$ \\ \hline
$\left((1,1),(1)\right)$ & $1+1=2$ & $4$ & $a_2=1,\;b_1=2$ \\ \hline
$\left((1),(1)\right)$ & $1=1$& $5$ & $a_1=1,\;b_1=1$\\ \hline
\end{tabular}
\end{center}
\label{tab1}
\end{table}
\begin{table}[ht]
\caption{Finding $\ld_{i,j}$ at level $k$ in reverse ordering.}
\begin{center}
\begin{tabular}{|p{4.0cm}| p{1.5cm} | p{1.5cm} | p{7.7cm} |}
\hline
\mbox{\bf Solution/Partition}  &\mbox{\bf Level $k$}& \mbox{\bf $a_s,\;b_t$}  &\mbox{\bf  $\quad\ld^{(k)}_{i,j}:=\ld_{i,j}$ at level $k$} \\  \hline 
\hline 	
$\left((1),(1)\right)$\vfill 
$1=1$ & $5$ & $a_1=1,\;\;$  $b_1=1$ & $\ld^{(5)}_{1,1}=\frac{1}{a_1}=1.$ \\ \hline
$\left((1,1),(1)\right)$ \vfill
$1+1=2$ 
& $4$  &  $a_2=1,\;\;$  $b_1=2$ & $\ld^{(4)}_{2,1}=\frac{1}{b_1}=\frac{1}{2}$; $\ld^{(4)}_{1,1}=\frac{b_1-a_2}{b_1}\cdot\ld^{(5)}_{1,1}=\frac{1}{2}.$ \\ \hline
$\left((1,1),(1,1)\right)$ \vfill
$3+1=2+2$ & $3$ &  $a_1=3,\;\;$ $b_2=2$ &    $\ld^{(3)}_{1,2}=\frac{1}{a_1}=\frac{1}{3}$; $\ld^{(3)}_{1,1}=\frac{a_1-b_2}{a_1}\cdot\ld^{(4)}_{1,1}=\frac{1}{6}$; $\qquad$  
$\ld^{(3)}_{2,1}=\ld^{(4)}_{2,1}=\frac{1}{2}$; $\ld^{(3)}_{2,2}=0$.  
\\ \hline
$\left((1,1),(1,1,1)\right)$\vfill
$6+1=2+2+3$ & $2$ &  $a_1=6,\;\;$ $b_3=3$ &  $\ld^{(2)}_{1,3}=\frac{1}{a_1}=\frac{1}{6}$; 
$\ld^{(2)}_{1,2}=\frac{a_1-b_3}{a_1}\cdot\ld^{(3)}_{1,2}=\frac{1}{6}$; $\qquad$  
$\ld^{(2)}_{1,1}=\frac{a_1-b_3}{a_1}\cdot\ld^{(3)}_{1,1}=\frac{1}{12}$; $\ld^{(2)}_{2,1}=\ld^{(3)}_{2,1}=\frac{1}{2}$; $\ld^{(2)}_{2,2}=0$; $\ld^{(2)}_{2,3}=0$.  \\\hline
$\left((1,1),(1,1,1,1)\right)$\vfill
$6+6=2+2+3+5$ & $1$ &  $a_2=6,\;\;$ $b_4=5$ &  $\ld^{(1)}_{2,4}=\frac{1}{a_2}=\frac{1}{6}$; $\ld^{(1)}_{2,3}=\frac{a_2-b_4}{a_2}\cdot\ld^{(2)}_{2,3}=0$; $\qquad$
$\ld^{(1)}_{2,2}=\frac{a_2-b_4}{a_2}\cdot\ld^{(2)}_{2,2}=0$; $\ld^{(1)}_{2,1}=\frac{a_2-b_4}{a_2}\cdot\ld^{(2)}_{2,1}=\frac{1}{12}$; $\ld^{(1)}_{1,1}=\ld^{(2)}_{1,1}=\frac{1}{12}$; $\ld^{(1)}_{1,2}=\ld^{(2)}_{1,2}=\frac{1}{6}$; 
  $\qquad\quad$ 
 $\ld^{(1)}_{1,3}=\ld^{(2)}_{1,3}=\frac{1}{6}$;
 $\ld^{(1)}_{1,4}=0$. \\\hline
\end{tabular}
\end{center}
\label{tab2}
\end{table}

\n By using the last row (and last column) of Table~\ref{tab2}, and setting $\a=(a_1,a_2)=(6,6)$ and 
$\b=(b_1,b_2,b_3,b_4)=(2,2,3,5)$, we verify that
\begin{align}\label{eq:f1}
((1,1),(1,1,1,1))
&=\sum_{i=1}^2\sum_{j=1}^4\ld^{(1)}_{i,j}\cdot \g_{i,j}\cr
&={\tiny\frac{1}{12}}\,\g_{1,1}+{\small\frac{1}{6}}\, \g_{1,2}+{\small\frac{1}{6}}\,\g_{1,3}+0\,\g_{1,4}+{\small\frac{1}{12}}\, \g_{2,1}+0,\g_{2,2}+0\,\g_{2,3}+{\small\frac{1}{6}}\,\g_{2,4}\cr
&=\frac{1}{12}\, ((2,0),(6,0,0,0))+\frac{1}{6}\, ((2,0),(0,6,0,0))+\frac{1}{6}\, ((3,0),(0,0,6,0))+\cr
&\quad\qquad\qquad\qquad\qquad\qquad\qquad\frac{1}{12}\, ((0,2),(6,0,0,0))+\frac{1}{6}\,((0,5),(0,0,0,6)).
\end{align}
We can also recover a linear combination for the non-binary solution  $((2),(2,1,1))$ to the original equation $x_1=2y_1+3y_2+5y_3$ as follows. From the binary solution 
$((1,1),(1,1,1,1))$, we recover the original solution 
as $((1+1),(1+1,1,1))=((2),(2,1,1))$, we perform corresponding ``moves'' on the pairs of vectors (generators) that appear in~\eqref{eq:f1}, i.e., add the first two coordinates of the first vector
in the solution-pair, and then add the last two coordinates of the second vector in the solution-pair. We summarize this in the following table.

\begin{table}[ht]
\caption{Recovering the linear combination of original solution}
\begin{center}
\begin{tabular}{ |p{5.8cm}| p{8.7cm} |}
\hline
\mbox{\bf Generator for binary case} &\mbox{\bf Generator for original (non-binary) case}\\  \hline 
\hline 	
$((2,0),(6,0,0,0))$ & $((2+0),(6+0,0,0))=((2),(6,0,0))$ \\ \hline
$((3,0),(0,6,0,0))$ & $((3+0),(0,6+0,0))=((3),(0,6,0))$ \\ \hline
$((0,2),(6,0,0,0))$ & $((0+2),(6+0,0,0))=((2),(6,0,0))$ \\ \hline
$((0,5),(0,0,0,6))$ & $((0+5),(0,0,0+6))=((5),(0,0,6))$ \\ \hline
\end{tabular}
\end{center}
\label{tab3}
\end{table}

From~\eqref{eq:f1} and Table~\ref{tab3}, we obtain 
\begin{align*}\label{eq:f2}
((2),(2,1,1))
&=((1+1),(1+1,1,1))\cr
&=\frac{1}{12}\cdot ((6),(2,0,0))+\frac{1}{6}\cdot ((3),(0,6,0))+\frac{1}{6}\cdot((2),(6,0,0))+\cr
&\qquad\qquad\qquad\qquad\qquad\qquad\qquad\frac{1}{12}\cdot ((2),(6,0,0))+\frac{1}{6}\cdot((5),(0,0,6))\cr
&=\frac{1}{3}\cdot ((2),(6,0,0))+
\frac{1}{6}\cdot((3),(0,6,0))+\frac{1}{6}\cdot ((5),(0,0,6)).
\end{align*}

\end{expl}

\bs\n{\bf Acknowledgement:}\
We are indebted to an anonymous referee for valuable suggestions which helped improve the presentation of the paper and correct an error in the previous version of Corollary~\ref{cor:2}.

\end{document}